\documentclass[11pt]{amsart}
\usepackage{amsmath}
\usepackage{amsfonts}
\usepackage{amsthm}
\usepackage{amssymb}
\usepackage{graphicx, url}

\newtheorem{theorem}{Theorem}[section]
\newtheorem{corollary}[theorem]{Corollary}
\theoremstyle{definition}

\begin{document}

\title
{Triangle-free Uniquely 3-Edge Colorable Cubic Graphs}

\author{sarah-marie belcastro}
\address{Department of Mathematics and Statistics, Smith College, Northampton, MA 01063 USA}
\email{smbelcas@toroidalsnark.net}
\author{Ruth Haas}
\address{Department of Mathematics and Statistics, Smith College, Northampton, MA 01063 USA}
\email{rhaas@smith.edu}

%\date{(date1), and in revised form (date2).}
\subjclass[2010]{05C10, 05C15}
\keywords{edge coloring, unique coloring, triangle-free, cubic graphs}

\thanks{Ruth Haas's work partially supported by NSF grant DMS-1143716 and Simons Foundation Award Number 281291}

\begin{abstract} This paper presents infinitely many new examples of  tri\-angle-free uniquely 3-edge colorable cubic graphs.   The only such graph previously known was given by Tutte in 1976.
\end{abstract}

\maketitle

\section{History}

Recall that a \emph{cubic} graph is 3-regular, that a \emph{proper 3-edge coloring} assigns colors to edges such that no two incident edges receive the same color, that \emph{edge-Kempe chains} are maximal sequences of edges that alternate between two colors, and that a \emph{Hamilton cycle} includes all vertices of a graph.

It is well known that a cubic graph with a Hamilton cycle is 3-edge colorable, as the Hamilton cycle is even (and thus 2-edge colorable) and its complement is a matching (that can be monochromatically colored).  A uniquely 3-edge colorable cubic graph must have exactly three Hamilton cycles, each an edge-Kempe chain in one of the ${3\choose 2}$ pairs of colors.  The converse is not true, as a cubic graph may have some colorings with Hamilton edge-Kempe chains and other colorings with non-Hamilton edge-Kempe chains; examples are given in \cite{Th82}.

The literature classifying uniquely 3-edge colorable cubic graphs is sparse; there is no complete characterization \cite{GY}.  It is well known that  the property of being uniquely 3-edge colorable is invariant under application of  $\Delta-Y$ transformations.  It was conjectured that  every simple planar cubic graph with exactly three Hamilton cycles contains a triangle \cite[Cantoni]{tutte}, and also that every simple planar uniquely 3-edge colorable cubic graph contains a triangle \cite{FW}.  The latter conjecture is proved in \cite{fow}, where it is also shown that if a simple planar cubic graph has exactly three Hamilton cycles, then it contains a triangle if and only if it is uniquely 3-edge colorable.  

Tutte, in a 1976 paper about the average number of Hamilton cycles in a graph \cite{tutte}, 
offhandedly remarks that one example of a nonplanar triangle-free uniquely 3-edge colorable cubic graph is the generalized Petersen graph $P(9,2)$, pictured in Figure \ref{fig:nine}.  He describes the graph as two 9-cycles $a_0\dots a_8$, $b_0\dots b_8$, with additional edges $a_ib_{2i}$ and index arithmetic done modulo 9.  \begin{figure}%[htp]
\begin{center}
\includegraphics[scale=.6]{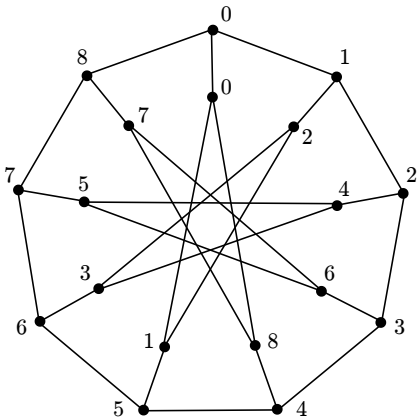} 
\caption{The generalized Petersen graph $P(9,2)$ labeled with Tutte's indices.}
\label{fig:nine}
\end{center}
\end{figure} 
(The generalized Petersen graph $P(m,2)$ is defined analogously, and in fact the known cubic graphs with exactly three Hamilton cycles and multiple distinct 3-edge colorings are $P(6k+3,2)$ for $k>1$ \cite{Th82}.)
It appears that the search for examples of triangle-free nonplanar uniquely 3-edge colorable cubic graphs ended with Tutte, or at least that any further efforts have been unsuccessful. Multiple sources (\cite{gcq2}, \cite{GY},  \cite{JT}) note that Tutte's example is the only known triangle-free nonplanar example.   It has been conjectured \cite{FW} that $P(9,2)$ is the only example. In Section \ref{sec:tri-free} we give infinitely many such graphs.

\section{New examples of triangle-free nonplanar uniquely 3-edge colorable cubic graphs}\label{sec:tri-free}

In \cite{bH} the authors introduced the following construction.  Consider two cubic graphs $G_1, G_2$, and form $G_1\thinspace\includegraphics[scale=.3]{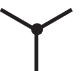}\thinspace G_2$ by choosing a vertex $v_i$ in $G_i\ (i=1,2)$, removing $v_i$ from $G_i\ (i=1,2)$, and adding a matching of three edges joining the three neighbors of $v_1$ with the three neighbors of $v_2.$
Of course there are many ways 
to choose $v_1, v_2$, and many ways to identify their incident edges, so the construction is not unique.  However, it is reversible; given a cubic graph $G$ with a 3-edge cut, we may decompose $G = G_1\thinspace\includegraphics[scale=.3]{Y.eps}
\thinspace G_2$. In that paper we proved the following.

\begin{theorem}\label{thm:old}[3.8 of \cite{bH}]
 Let $G_1, G_2$ be cubic graphs and $a_i$ the number of 3-edge colorings of  $G_i$.  Then $G_1\thinspace\includegraphics[scale=.3]{Y.eps} \thinspace G_2$  has  $a_1a_2$ edge colorings.\end{theorem}

Define $G^{\includegraphics[scale=.3]{Y.eps}}$ to be the infinite family of graphs consisting of all  graphs of the form $G \thinspace\includegraphics[scale=.3]{Y.eps}\thinspace G \thinspace\includegraphics[scale=.3]{Y.eps}\cdots  \includegraphics[scale=.3]{Y.eps}\thinspace G$.
This leads to the following corollaries of  Theorem \ref{thm:old}.

%%%%%%%%%%%%%%%

\begin{theorem}\label{thm:yup} If $G$ is a  uniquely 3-edge colorable graph, then all graphs in   $G^{\includegraphics[scale=.3]{Y.eps}}$ are uniquely 3-edge colorable.
 \end{theorem}
 
 \begin{proof}   The proof proceeds by induction on the number of copies of $G$.   
 % We consider $G \thinspace\includegraphics[scale=.3]{Y.eps}\thinspace G$, which by construction has a 3-edge cut.  By \cite[Theorem 3.4]{bH}, any 3-edge coloring $c$ of $G \thinspace\includegraphics[scale=.3]{Y.eps}\thinspace G$ can be written as $c_1\thinspace\includegraphics[scale=.3]{Y.eps}\thinspace c_2$, where $c_1, c_2$ are 3-edge colorings of $G$.  Of course, there is only one such coloring, and because the 3-edge cut uses all three colors \cite[Parity Lemma]{isaacs}, this forces a unique on  3-edge coloring on $G \thinspace\includegraphics[scale=.3]{Y.eps}\thinspace G$.  

%Now consider $G \thinspace\includegraphics[scale=.3]{Y.eps}^k$, which by construction has several 3-edge cuts.  By inductive hypothesis and application of \cite[Theorem 3.4]{bH}, the desired result is achieved.
\end{proof}

\begin{corollary}\label{ex:yes}
All members of the infinite family  
$P(9,2)^{\includegraphics[scale=.3]{Y.eps}}$

are uniquely 3-edge colorable.

\end{corollary}

%Note that $P(9,2)$ is known to be uniquely 3-edge colorable, so that direct application of Theorem  \ref{thm:yup} gives the desired result.

{\bf Note.}

 In  \cite{gcq1}, Goldwasser and Zhang proved that if a uniquely 3-edge colorable graph has an edge cut of size  3 or 4 such that each remaining component contains a cycle, then the graph can be decomposed into two smaller uniquely 3-edge colorable graphs.  It seems they did not observe the reverse construction.

\subsection{ Examples and Properties }
The smallest member of $P(9,2)^{\includegraphics[scale=.3]{Y.eps}}$  is of course $P(9,2)$, which  has $18$ vertices. 
 For  every  integer $k>1$ there are multiple graphs  in $P(9,2)^{\includegraphics[scale=.3]{Y.eps}}$ with $16k+2$ vertices. 
Nonisomorphic examples with $k=2$ are shown in Figures  \ref{fig:doublenine} and \ref{fig:doublecross}.

\begin{figure}%[htp]
\begin{center}
\includegraphics[scale=.45]{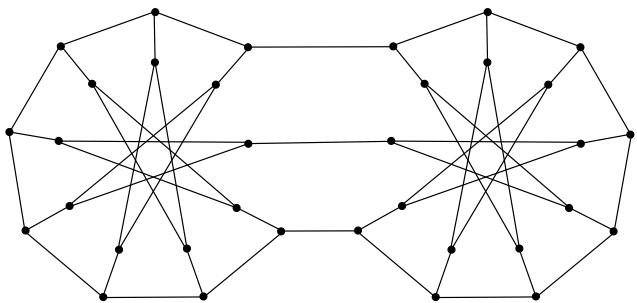} 
\caption{A nonplanar, triangle-free, uniquely 3-edge colorable graph with 34 vertices.}
\label{fig:doublenine}
\end{center}
\end{figure}
%As an extra precaution, we verified that $P(9,2) \thinspace\includegraphics[scale=.3]{Y.eps}\thinspace P(9,2)$ is uniquely 3-edge colorable by brute force using custom \emph{Mathematica} code.  

\begin{figure}%[htp]
\begin{center}
\includegraphics[scale=.45]{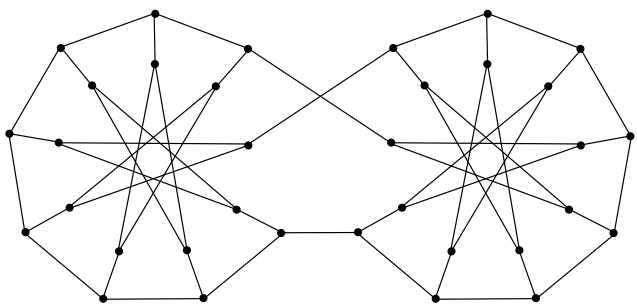} 
\caption{A nonplanar, triangle-free, uniquely 3-edge colorable graph with 34 vertices that is nonisomorphic to that shown in Figure \ref{fig:doublenine}.}
\label{fig:doublecross}
\end{center}
\end{figure}

The  graphs in  $P(9,2)^{\includegraphics[scale=.3]{Y.eps}}$ are clearly all nonplanar.   We show next that there are graphs in $P(9,2)^{\includegraphics[scale=.3]{Y.eps}}$ of every other orientable and nonorientable genus.

\begin{theorem}  Every graph in $P(9,2)^{\includegraphics[scale=.3]{Y.eps}}$  with $16k+2$ vertices has orientable and nonorientable genus at most $k$.
Further, there is a large subfamily of graphs in $P(9,2)^{\includegraphics[scale=.3]{Y.eps}}$, each of which has $16k+2$ vertices
and orientable and nonorientable genus exactly $k$.
\end{theorem}

 \begin{figure}
\begin{center}
\includegraphics[scale=.35]{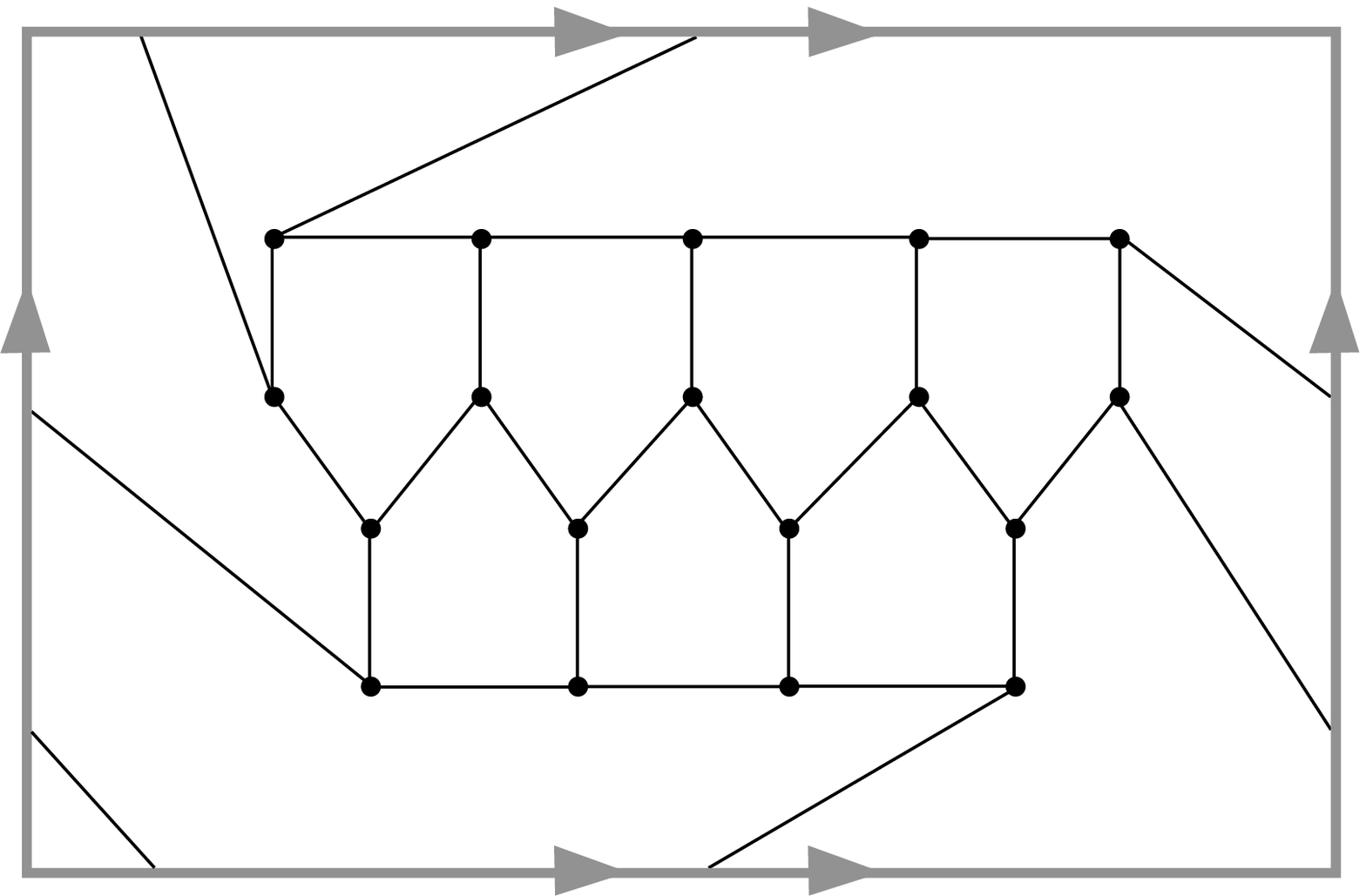} \hspace{.5cm} \raisebox{-0.15\height}{  \includegraphics[scale=.4]{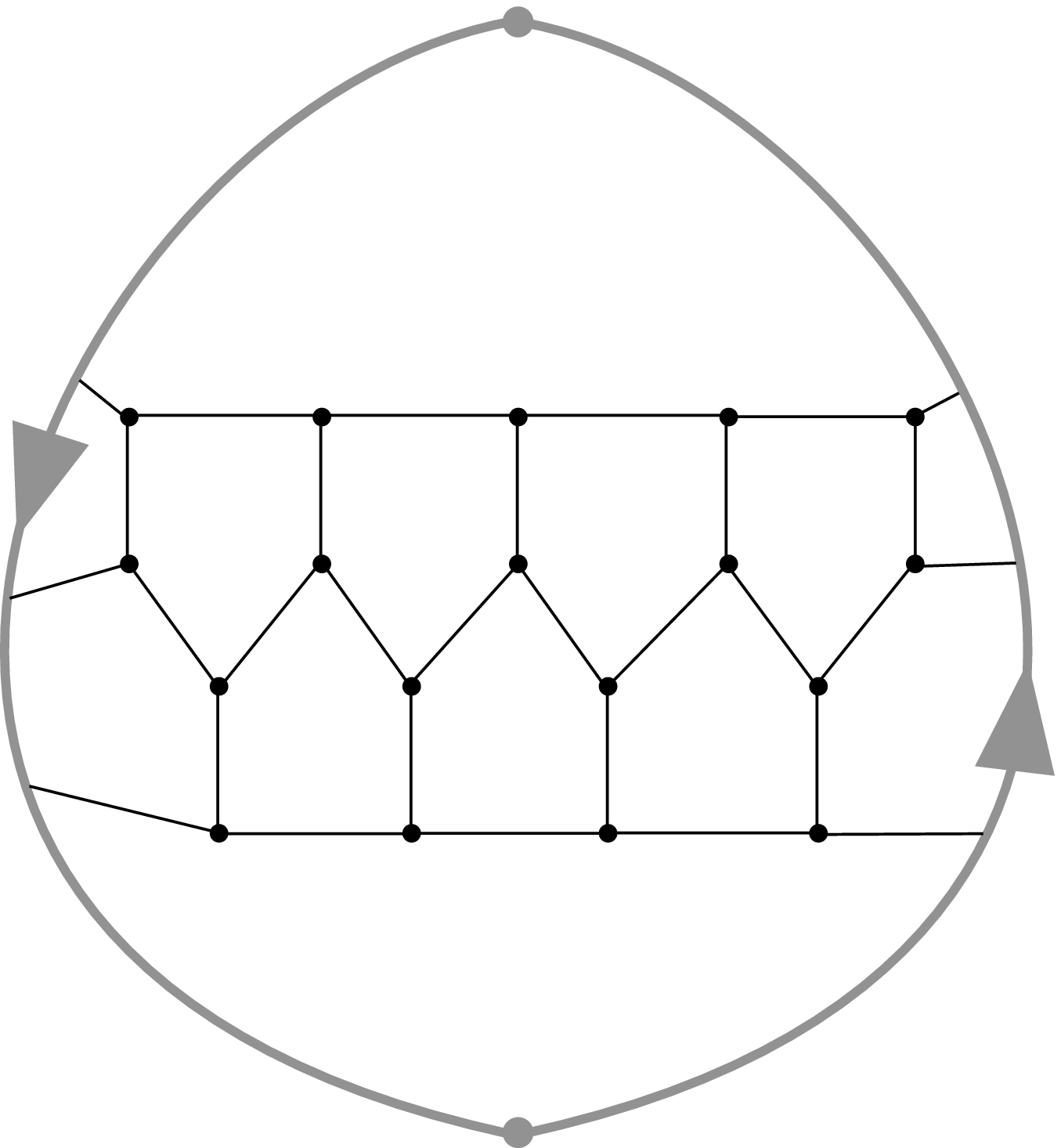} }
\caption{Embeddings of $P(9,2)$ on the torus (left) and  projective plane (right) } 
\label{fig:tor-pp}
\end{center}
\end{figure}

\begin{proof}
 We will show by induction that any graph $Q_k$ created using the \includegraphics[scale=.3]{Y.eps}-construction with $k$ copies of $P(9,2)$ has orientable and nonorientable genus at most $k$.  The base case holds because $P(9,2)$ embeds on both the torus (see Figure \ref{fig:tor-pp} (left)) and on the projective plane (see Figure \ref{fig:tor-pp} (right)).

Now consider $Q_{k}$, a graph created using the \includegraphics[scale=.3]{Y.eps}-construction with $k$ copies of $P(9,2)$.
 $Q_k$ was obtained  by removing and associating $v\in P(9,2)$ and some
 $ w\in Q_{k-1}$ via the \includegraphics[scale=.3]{Y.eps}-construction, where $Q_{k-1}$ is some graph created using  $k-1$ copies of $P(9,2)$ that has genus $k-1$  or less by inductive hypothesis.
 Let $\widehat{Q_k}$ be the 
 graph produced by simply identifying the vertices $v$ and $w$. The graph $\widehat{Q_k}$ has two
 blocks that meet at this vertex and by  Theorem 1 of  \cite{bhky}, the genus of  $\widehat{Q_k}$  is the 
 sum of the genera of the blocks (so it is $k$).  
Replacing the cut vertex by a 3-edge cut to implement the \includegraphics[scale=.3]{Y.eps}
construction does not increase the genus, which completes the proof of the upper bound on genus.

\begin{figure}[htp]
\begin{center}
\includegraphics[scale=.4]{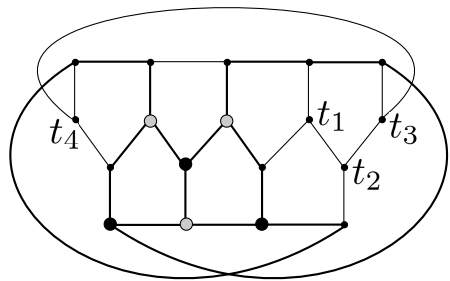} 
\caption{$P(9,2)$ with a copy of a subdivision of $K_{3,3}$ highlighted.}
\label{fig:pent}
\end{center}
\end{figure}

A copy of a subdivision of $K_{3,3}$  is highlighted in the embedding of $P(9,2)$ shown in Figure \ref{fig:pent}. 
There are  four vertices $\{ t_1, t_2, t_3, t_4\}$ whose edges are not involved in the subdivided $K_{3,3}$.  
Any (or all) of $\{ t_1, t_2, t_3, t_4\}$ can be removed 
and the resulting graph will still have a $K_{3,3}$ minor.   If $Q_k$ is formed such that in each copy of 
$P(9,2)$ only (some) of vertices $\{ t_1, t_2, t_3, t_4\}$ are used in the 
 \includegraphics[scale=.3]{Y.eps} construction then there will still be $k$ disjoint copies of subdivisions 
 of  $K_{3,3}$ in $Q_k$. The genus of a graph is the sum of the genera of its components 
 \cite[Cor. 2]{bhky}, so using this construction $Q_k$  has a minor with orientable
  (resp. nonorientable) genus  exactly $k$. It is straightforward to draw an embedding of sample $Q_k$ on a surface of orientable
 or nonorientable genus $k$.

%{\it  Turns out the $t_i$ have to be  a path on one induced 9 cycle of  $P(9,2)$ so there are limited choices.}

 \end{proof}

\section{Conclusion}

While we have provided infinitely many examples of triangle-free nonplanar uniquely 3-edge colorable cubic graphs, it is still unknown whether other examples exist. 
All our examples support  Zhang's conjecture \cite{cq} that every triangle-free uniquely 3-edge colorable cubic graph contains a Petersen graph minor.   That conjecture remains open.

\nocite{*}
\bibliographystyle{amsplain}
\bibliography{}

\end{document}